\documentclass[12pt,reqno]{amsart}
\usepackage{amssymb}
\usepackage[all]{xy}

\newtheorem{theorem}{Theorem}[section]

\newtheorem{lemma}[theorem]{Lemma}

\numberwithin{equation}{section}

\begin{document}
 	
\baselineskip=15pt

\title[Effective cone of a Grassmann bundle over a curve]{Effective cone of a Grassmann bundle over
a curve defined over $\overline{\mathbb F}_p$}

\author[I. Biswas]{Indranil Biswas}

\address{Department of Mathematics, Shiv Nadar University, NH91, Tehsil
Dadri, Greater Noida, Uttar Pradesh 201314, India}

\email{indranil.biswas@snu.edu.in, indranil29@gmail.com}

\author[S. M. Garge]{Shripad M. Garge}

\address{Department of Mathematics, Indian Institute of Technology Bombay,
Powai, Mumbai 400076, Maharashtra, India}

\email{shripad@math.iitb.ac.in, smgarge@gmail.com}

\author[K. Hanumanthu]{Krishna Hanumanthu}

\address{Chennai Mathematical Institute, H1 SIPCOT IT Park, Siruseri, Kelambakkam 603103, India}

\email{krishna@cmi.ac.in}



\date{}

\begin{abstract}
Let $X$ be an irreducible smooth projective curve defined over $\overline{\mathbb F}_p$ and $E$ a vector bundle on
$X$ of rank at least two. For any $1\, \leq\, r\, <\, {\rm rank}(E)$, let ${\rm Gr}_r(E)$ be the Grassmann bundle over $X$ 
parametrizing all the $r$ dimensional quotients of the fibers of $E$. 
We prove that the effective cone in ${\rm NS}({\rm Gr}_r(E))\otimes_{\mathbb Z} {\mathbb R}$ coincides
with the pseudo-effective cone in ${\rm NS}({\rm Gr}_r(E))\otimes_{\mathbb Z} {\mathbb R}$. When $r\,=\,1$ or
${\rm rank}(E)-1$, this was proved in \cite{Mo}.
\end{abstract}

\maketitle
 	
\section{introduction} 

Let $X$ be an irreducible smooth projective curve defined over an algebraically closed field $k$, and let $E$ be a vector bundle over
$X$ of rank $N$, with $N\, >\, 1$. Fix an integer $1\,\leq\, r\, \leq\, N-1$, and denote by ${\rm Gr}_r(E)$ the Grassmann bundle over $X$
parametrizing all the $r$ dimensional quotients of the fibers of $E$. The N\'eron--Severi group ${\rm NS}({\rm Gr}_r(E))$ is
the group of divisors on ${\rm Gr}_r(E)$ modulo algebraic equivalence, so ${\rm NS}({\rm Gr}_r(E))$ coincides with the
group of connected components of the Picard group ${\rm Pic}({\rm Gr}_r(E))$. The pseudo-effective cone of ${\rm NS}
({\rm Gr}_r(E))_{\mathbb R}\, :=\, {\rm NS}({\rm Gr}_r(E))\otimes_{\mathbb Z} {\mathbb R}$ is the closure of the effective
cone of ${\rm NS}({\rm Gr}_r(E))_{\mathbb R}$. In \cite{BHP} the pseudo-effective cone of ${\rm NS}({\rm Gr}_r(E))_{\mathbb R}$
was computed.

Let $p$ be a prime number, and let $\mathbb{F}_p$ be the field of order $p$. 
When $k$ is the algebraic closure $\overline{\mathbb F}_p$ of $\mathbb{F}_p$, Moriwaki proved that the 
pseudo-effective cone of ${\rm NS}({\rm Gr}_1(E))_{\mathbb R}$ coincides with the effective cone of ${\rm NS}({\rm Gr}_1(E))_{\mathbb R}$
\cite[p.~802, Theorem 0.4]{Mo}. Replacing $E$ by its dual $E^*$ it is deduced from this that the 
pseudo-effective cone of ${\rm NS}({\rm Gr}_{N-1}(E))_{\mathbb R}$ coincides with the effective cone of
${\rm NS}({\rm Gr}_{N-1}(E))_{\mathbb R}$.

Our aim here is to prove the following (see Theorem \ref{prop1}):

\noindent
\textbf{Theorem:} 
{\it Let $k \,=\,  \overline{\mathbb F}_p$ for a prime $p$. The pseudo-effective cone of ${\rm NS}({\rm Gr}_r(E))_{\mathbb R}$
coincides with the effective cone of ${\rm NS}({\rm Gr}_r(E))_{\mathbb R}$ for all $1\,\leq\, r\, \leq\, N-1$.}

\section{Pseudo-effective cone of a Grassmann bundle}

Let $k$ be an algebraically closed field. Let $X$ be an irreducible smooth projective curve defined over $k$.
Take any vector bundle $E$ over $X$ such that
\begin{equation}\label{e0}
N\,:=\, {\rm rank}(E)\, \geq\, 2.
\end{equation}
Fix an integer $1\,\leq\, r\, \leq\, N-1$. Let
\begin{equation}\label{e1}
\phi\,\,:\,\, {\rm Gr}_r(E)\,\, \longrightarrow\,\, X
\end{equation}
be the Grassmann bundle parametrizing the $r$ dimensional quotients of the fibers of $E$. Let
${\mathcal O}_{{\rm Gr}_r(E)}(1)$ be the tautological relatively ample line bundle over
${\rm Gr}_r(E)$. The fiber of ${\mathcal O}_{{\rm Gr}_r(E)}(1)$
over the point of ${\rm Gr}_r(E)$ representing a quotient $E_x \, \longrightarrow\, Q$, where $x\, \in\, X$, is $\bigwedge^r Q$.
Fix a line bundle $L$ on $X$ of degree one. The N\'eron--Severi group 
${\rm NS}({\rm Gr}_r(E))$ of ${\rm Gr}_r(E)$, which is
the group of connected components of the Picard group ${\rm Pic}({\rm Gr}_r(E))$, is the free
abelian group generated by the classes of ${\mathcal O}_{{\rm Gr}_r(E)}(1)$ and $\phi^*L$,
where $\phi$ is the projection in \eqref{e1}; this follows from the Seesaw
Theorem \cite[p.~54, Corollary 6]{Mu}. Denote
\begin{equation}\label{e2}
{\rm NS}({\rm Gr}_r(E))_{\mathbb R}\,\, :=\,\, {\rm NS}({\rm Gr}_r(E))
\otimes_{\mathbb Z} {\mathbb R}.
\end{equation}
A cone in ${\rm NS}({\rm Gr}_r(E))_{\mathbb R}$ is a convex subset of it closed under the multiplication by nonnegative real numbers.

The effective cone of ${\rm Gr}_r(E)$ is the cone in ${\rm NS}({\rm Gr}_r(E))_{\mathbb R}$ (defined
in \eqref{e2}) generated by the effective divisors. The pseudo-effective cone of ${\rm Gr}_r(E)$ is the closure,
in ${\rm NS}({\rm Gr}_r(E))_{\mathbb R}$, of the effective cone.
The pseudo-effective cone of ${\rm Gr}_r(E)$ was computed in \cite{BHP}, which will be briefly recalled.

\subsection{When the characteristic is zero}

First assume that the characteristic of $k$ is zero.

Let
$$
0\,=\, E_0\, \subset\, E_1\, \subset\, \cdots\, \subset\, E_{m-1}\, \subset\, E_m\,=\, E
$$
be the Harder--Narasimhan filtration of $E$ \cite[p.~16, Theorem 1.3.4]{HL}. Let $1\,\leq\,
\ell\, \leq\, m$ be the unique integer such that
$\text{rank}(E_{\ell-1})\, <\, r \, \leq\, {\rm rank}(E_\ell)$. Define
\begin{equation}\label{e4}
\lambda\,\,:=\,\, \text{degree}(E_{\ell-1}) + (r- \text{rank}(E_{\ell-1}))\mu(E_\ell/E_{\ell-1}),
\end{equation}
where $\mu(F)\,:=\, \frac{\text{degree}(F)}{{\rm rank}(F)}$.

\begin{lemma}[{\cite[p.~74, Theorem 4.1]{BHP}}]\label{lem1}
The pseudo-effective cone of ${\rm Gr}_r(E)$ is generated by $\phi^*c_1(L)$ and
$c_1({\mathcal O}_{{\rm Gr}_r(E)}(1))- \lambda \phi^*c_1(L)$, where $\lambda$ is defined in \eqref{e4}.
\end{lemma}

\subsection{When the characteristic is positive}

Assume that the characteristic of $k$ is $p$, with $p\, >\, 0$.

For any vector bundle $W$ on $X$, we have the vector bundle $F^*_XW$ on $X$, where $F_X$ is the absolute Frobenius
morphism of $X$. We recall that $F^*_X W$ is the subbundle of $W^{\otimes p}$ defined by the image of the morphism
$W\, \longrightarrow\, W^{\otimes p}$ that sends any $w\, \in\, W$ to $w^{\otimes p}$. For any $j\, \geq\, 1$, the $j$--fold
iteration of $W\, \longmapsto\, F^*_XW$ will be denoted by $(F^j_X)^*W$; by $(F^0_X)^*W$ we will denote $W$.

For any $j \, \geq\, 0$, let
$$
0\,=\, E_{j,0}\, \subset\, E_{j,1}\, \subset\, \cdots\, \subset\, E_{j,n_j}\,=\, (F^j_X)^*E
$$
be the Harder--Narasimhan filtration of $(F^j_X)^*E$. There is a nonnegative integer $\delta\,=\, \delta(E)$ such that
$$
0\,=\, (F^j_X)^*E_{\delta,0}\, \subset\, (F^j_X)^*E_{\delta,1}\, \subset\, \cdots\, \subset\, (F^j_X)^*E_{\delta,n_\delta}\,=\,
(F^j_X)^*(F^\delta_X)^*E\,=\, (F^{\delta+j}_X)^*E
$$
is the Harder--Narasimhan filtration of $(F^{\delta+j}_X)^*E$ for all $j\, \geq\, 0$;
so $n_\delta\,=\, n_{\delta+j}$ and
$(F^j_X)^*E_{\delta,i}\,=\, E_{\delta+j,i}$ for all $1\, \leq\, i\,\leq\, n_\delta$.
Note that if $\delta$ satisfies the above
condition, then any integer greater than $\delta$ also satisfies the above condition. Also, if $\delta$ satisfies the above
condition, then $E_{\delta,i}/E_{\delta,i-1}$ is strongly semistable for all
$1\, \leq\, i\, \leq\, n_\delta$.

Fix a $\delta$ satisfying the above condition.
Let $1\,\leq\, \ell\, \leq\, n_\delta$ be the unique integer such that
$\text{rank}(E_{\delta, \ell-1})\, <\, r \, \leq\, {\rm rank}(E_{\delta,\ell})$.
Define
\begin{equation}\label{e5}
\lambda\,\,:=\,\, \frac{1}{p^\delta}\left(\text{degree}(E_{\delta,\ell-1}) +
(r- \text{rank}(E_{\delta,\ell-1}))\mu(E_{\delta,\ell}/E_{\delta, \ell-1})\right).
\end{equation}
Note that $\lambda$ does not depend on the choice of $\delta$.

\begin{lemma}[{\cite[p.~76, Theorem 4.4]{BHP}}]\label{lem2}
The pseudo-effective cone of ${\rm Gr}_r(E)$ is generated by $\phi^*c_1(L)$ and
$c_1({\mathcal O}_{{\rm Gr}_r(E)}(1))- \lambda \phi^*c_1(L)$, where $\lambda$ is defined in \eqref{e5}.
\end{lemma}

\section{The effective cone}

Set $k\,=\, \overline{\mathbb F}_p$, with $p\,>\, 0$.

The following proposition describes the effective cone of ${\rm Gr}_r(E)$ contained in the pseudo-effective cone of ${\rm Gr}_r(E)$.

\begin{theorem}\label{prop1}
The effective cone of ${\rm Gr}_r(E)$ coincides with the pseudo-effective cone of ${\rm Gr}_r(E)$.
\end{theorem}

\begin{proof}
The pseudo-effective cone of ${\rm Gr}_r(E)$ is described in Lemma \ref{lem2}. We need to show that the
two boundary edges are contained in the effective cone of ${\rm Gr}_r(E)$. 

The class $\phi^*c_1(L)$ in Lemma \ref{lem2} is given by a fiber of the map $\phi$ in \eqref{e1}. So it suffices
to show that the class $c_1({\mathcal O}_{{\rm Gr}_r(E)}(1))- \lambda \phi^*c_1(L)$ in Lemma \ref{lem2} lies in the
effective cone.

Fix a pair
\begin{equation}\label{e6}
(Y,\, \Phi),
\end{equation}
where $Y$ is an irreducible smooth projective curve and
$$
\Phi\,\, :\,\, Y \, \longrightarrow\, X
$$
is a dominant morphism of such that there is a line bundle ${\mathcal L}_0$ on $Y$ satisfying the condition that
\begin{equation}\label{e7}
\Phi^*E\,\,=\,\, \bigoplus_{i=1}^N {\mathcal L}^{\otimes a_i}_0
\end{equation}
(see \eqref{e0}), where $a_i$ are integers; see \cite[p.~214, Proposition 2.1]{BP} for the existence of a pair
as in \eqref{e6}, \eqref{e7} (see also \cite[p.~809, Theorem 2.2]{Mo}).

Consider the vector bundle $$(F^{\delta}_Y)^*\Phi^*E \,\,=\,\, \Phi^* (F^\delta_X)^*E$$ on $Y$ (see \eqref{e5}), where $F_Y$
is the absolute Frobenius morphism of $Y$. Denoting $(F^{\delta}_Y)^*{\mathcal L}_0\,=\,
{\mathcal L}^{\otimes p^\delta}_0$ by $\mathcal L$, from
\eqref{e7} we have
\begin{equation}\label{e8}
{\mathcal E}\,\,:=\,\,
\Phi^* (F^\delta_X)^*E\,\,=\,\, \bigoplus_{i=1}^N (F^\delta_Y)^*{\mathcal L}^{\otimes a_i}_0
\,\,=\,\, \bigoplus_{i=1}^N {\mathcal L}^{\otimes a_i}.
\end{equation}

Recall that $E_{\delta,j}/E_{\delta,j-1}$ is strongly semistable for all $1\, \leq\, j\, \leq\, n_\delta$. This implies
that $\Phi^*(E_{\delta,j}/E_{\delta,j-1})\,=\, (\Phi^*E_{\delta,j})/(\Phi^*E_{\delta,j-1})$ 
is strongly semistable for every $1\, \leq\, j\, \leq\, n_\delta$. Consequently,
\begin{equation}\label{e9}
0\,=\, \Phi^*E_{\delta,0}\, \subset\, \Phi^* E_{\delta,1}\, \subset\, \cdots\, \subset\, \Phi^*E_{\delta,n_\delta}\,=\,
\Phi^*(F^\delta_X)^*E\,=\, {\mathcal E}
\end{equation}
is the Harder--Narasimhan filtration of $\mathcal E$ (defined
in \eqref{e8}). From \eqref{e8} and \eqref{e9} it follows that each
subbundle $\Phi^* E_{\delta,j}\, \subset\, {\mathcal E}$ in \eqref{e9} is a direct sum of some of the direct summands
in \eqref{e8}. Let $\sigma$ be a permutation of $\{1,\, \cdots,\, N\}$ such that
$$
\Phi^*E_{\delta, j}\,\,=\,\, \bigoplus_{i=1}^{{\rm rank}(E_{\delta, j})} {\mathcal L}^{\otimes a_{\sigma(i)}}
$$
for all $1\, \leq\, j\, \leq\, n_\delta$. Hence
\begin{equation}\label{e10}
\widetilde{\mathcal L}\,\,:=\,\,
\left(\bigotimes_{i=1}^n {\mathcal L}^{\otimes a_{\sigma(i)}}\right)\otimes
\left({\mathcal L}^{\otimes a_{\sigma(n+1)}}\right)^{\otimes (r-n)}
\end{equation}
is a direct summand of
$$
\bigwedge\nolimits^r {\mathcal E}\,\,=\,\, \bigwedge\nolimits^r \Phi^*(F^\delta_X)^*E,
$$
where $n\,=\, {\rm rank}(E_{\delta, \ell-1})$ (see \eqref{e5} for $\ell$).
This implies that
\begin{equation}\label{e11}
{\mathcal O}_Y \,\, \subset\,\, \left(\bigwedge\nolimits^r {\mathcal E}\right)\otimes \left(\widetilde{\mathcal L}\right)^*,
\end{equation}
where $\widetilde{\mathcal L}$ is the line bundle in \eqref{e10} and 
$\mathcal E$ is the vector bundle in \eqref{e8}.

{}From \eqref{e8} and \eqref{e9} we deduce that
\begin{equation}\label{c1}
\text{degree}\left(\widetilde{\mathcal L}\right)\,\,=\,\, \lambda\cdot \text{degree}
(\Phi\circ F^{\delta}_Y) \,\,=\,\, \lambda\cdot \text{degree}(F^\delta_X\circ\Phi),
\end{equation}
where $\lambda$ is defined in \eqref{e5}.

Let
\begin{equation}\label{e12}
\varphi\,\, :\,\, {\rm Gr}_r({\mathcal E})\,\, \longrightarrow\,\, Y
\end{equation}
be the Grassmann bundle parametrizing the $r$ dimensional quotients of the fibers of the
vector bundle $\mathcal E$ in \eqref{e8}. Let
$${\mathcal O}_{{\rm Gr}_r({\mathcal E})}(1) \,\, \longrightarrow\, \, {\rm Gr}_r({\mathcal E})$$
be the tautological line bundle whose fiber over the point of ${\rm Gr}_r(E)$ representing a quotient ${\mathcal E}_y \,
\longrightarrow\, Q$, where $y\, \in\, Y$, is $\bigwedge^r Q$. We have a Cartesian diagram
\begin{equation}\label{e13}
\begin{matrix}
{\rm Gr}_r({\mathcal E})& \xrightarrow{\,\,\,\Psi\,\,\,} & {\rm Gr}_r(E)\\
\,\,\, \Big\downarrow\varphi && \,\,\, \Big\downarrow\phi\\
Y & \xrightarrow{\,\,\,F^\delta_X\circ\Phi\,\,\,} & X
\end{matrix}
\end{equation}
(see \eqref{e12}, \eqref{e1} and \eqref{e6}). We have
\begin{equation}\label{c2}
\Psi^* {\mathcal O}_{{\rm Gr}_r(E)}(1)\,\,=\,\, {\mathcal O}_{{\rm Gr}_r({\mathcal E})}(1),
\end{equation}
where $\Psi$ is the map in \eqref{e13}.

Combining \eqref{c1} and \eqref{c2} it follows that
\begin{equation}\label{e14}
{\Psi}^*(c_1({\mathcal O}_{{\rm Gr}_r(E)}(1))- \lambda \phi^*c_1(L))\,\,=\,\,
c_1({\mathcal O}_{{\rm Gr}_r({\mathcal E})}(1)) - \varphi^* c_1(\widetilde{\mathcal L}),
\end{equation}
where $\widetilde{\mathcal L}$ is the line bundle in \eqref{e10}. By the projection formula,
$$
H^0\left({\rm Gr}_r({\mathcal E}),\, {\mathcal O}_{{\rm Gr}_r({\mathcal E})}(1)
\otimes \varphi^*\left(\widetilde{\mathcal L}\right)^*\right)\,\,=\,\,
H^0\left(Y,\, \left(\bigwedge\nolimits^r{\mathcal E}\right)\otimes \left(\widetilde{\mathcal L}\right)^*\right),
$$
and hence from \eqref{e11} we conclude that $H^0\left({\rm Gr}_r({\mathcal E}),\, {\mathcal O}_{{\rm Gr}_r({\mathcal E})}(1)
\otimes \varphi^*\left(\widetilde{\mathcal L}\right)^*\right)\,\not=\, 0$. This implies that the N\'eron--Severi class of
${\mathcal O}_{{\rm Gr}_r({\mathcal E})}(1)\otimes \varphi^*\left(\widetilde{\mathcal L}\right)^*$ is effective. 

Note that the pull-back map
$\Psi^*\,:\, {\rm NS}({\rm Gr}_r(E))_{\mathbb R} \,\longrightarrow\, {\rm NS}({\rm Gr}_r({\mathcal E}))_{\mathbb R}$
induced by $\Psi$ is an isomorphism. Therefore, from \eqref{e14}
it follows that the class $c_1({\mathcal O}_{{\rm Gr}_r(E)}(1))- \lambda \phi^*c_1(L)\, \in\,
{\rm NS}({\rm Gr}_r(E))_{\mathbb R}$ is effective. This completes the proof.
\end{proof}

\section*{Acknowledgements} A part of this work was done when the first and the third authors visited the Department of Mathematics, IIT 
Bombay. They are grateful for its hospitality. The third author was partially supported by a grant from Infosys Foundation.


\begin{thebibliography}{ZZZZZ}

\bibitem[BP]{BP} I. Biswas and A. J. Parameswaran,  On vector bundles on curves over $\overline{\mathbb F}_p$,
{\it Com. Ren. Math. Acad. Sci. Paris} {\bf 350} (2012), 213--216.

\bibitem[BHP]{BHP} I. Biswas, A. Hogadi and A. J. Parameswaran, Pseudo-effective cone of Grassmann
bundles over a curve, {\it Geom. Dedicata} {\bf 172} (2014), 69--77.

\bibitem[HL]{HL} D. Huybrechts and M. Lehn, {\it The geometry of moduli spaces of sheaves}, Aspects
of Mathematics, E31, Friedr. Vieweg~\&~Sohn, Braunschweig, 1997.

\bibitem[Mo]{Mo} A. Moriwaki, Toward a geometric analogue of Dirichlet's unit theorem,
{\it Kyoto Jour. Math.} {\bf 55} (2015), 799--817.

\bibitem[Mu]{Mu} D. Mumford, {\it Abelian varieties}, Tata Inst. Fundam. Res. Stud. Math., 5,
Published for the Tata Institute of Fundamental Research, Bombay by Oxford University
Press, London, 1970.

\end{thebibliography}
\end{document}